\documentclass[10pt,reqno]{amsart}
\usepackage[top=1.2in, left=1.1in, right=1.1in, bottom=.9in]{geometry}

\usepackage{bbm}
\usepackage[draft]{optional}
\usepackage[british,american]{babel}

\usepackage{mathrsfs}
\usepackage{amsfonts, amsmath, wasysym}
\usepackage{comment}

\usepackage{enumitem}
\usepackage{amssymb,amsthm}
\usepackage{nccmath}
\usepackage{MnSymbol}

\usepackage{todonotes}
\newcounter{todocounter}

\usepackage{
latexsym,
nicefrac,
}
\usepackage{url}

\definecolor{darkgreen}{rgb}{0,0.5,0}
\definecolor{darkred}{rgb}{0.7,0,0}
\usepackage[colorlinks, 
citecolor=darkgreen, linkcolor=darkred,urlcolor=black
]{hyperref}
\usepackage{esint}
\usepackage{bibgerm}
\usepackage[normalem]{ulem}

\theoremstyle{plain}
\newtheorem{lemma}{Lemma}[section]
\newtheorem{thm}[lemma]{Theorem}
\newtheorem{prop}[lemma]{Proposition}

\theoremstyle{definition}
\newtheorem{defn}[lemma]{Definition}

\newtheorem{rmk}[lemma]{Remark}
\numberwithin{equation}{section}

\newcommand{\HH}{\mathcal{H}}
\newcommand{\al}{\alpha}

\newcommand{\ga}{\gamma}

\newcommand{\de}{\delta}
\newcommand{\om}{\omega}

\newcommand{\Si}{\Sigma}

\newcommand{\D}{\mathbb{D}}
\newcommand{\La}{\Lambda}
\newcommand{\R}{\ensuremath{{\mathbb R}}}

\newcommand{\osc}{\text{osc}}

\newcommand{\upto}{\uparrow}

\DeclareMathSymbol{\intprod}{\mathbin}{MnSyC}{'270}
\newcommand{\T}{{\mathcal T}}
\newcommand{\That}{{ \widehat{\T} }}

\newcommand{\norm}[1]{\Vert#1\Vert}

\newcommand{\beq}{\begin{equation}}
\newcommand{\eeq}{\end{equation}}
\newcommand{\beqs}{\begin{equation*}}
\newcommand{\eeqs}{\end{equation*}}
\newcommand{\beqa}{\begin{equation}\begin{aligned}}
\newcommand{\eeqa}{\end{aligned}\end{equation}}
\newcommand{\beqas}{\begin{equation*}\begin{aligned}}
\newcommand{\eeqas}{\end{aligned}\end{equation*}}
\newcommand{\half}{\frac{1}{2}}

\newcommand{\leqs}{\lesssim}
\newcommand{\geqs}{\gtrsim}

\newcommand{\abs}[1]{\vert#1\vert} 

\newcommand{\eps}{\varepsilon}
\newcommand{\na}{\nabla}

\newcommand{\tr}{\text{tr}}
 
\newcommand{\Loj}{{\L}ojasiewicz }

\title{Topping's conjecture for solutions of harmonic map flow with isolated bubbles}
\date{\today}
\author{Melanie Rupflin}

\author{Alex Waldron}

\author{Sebastian Woodward}

\begin{document}

\begin{abstract}
    By establishing a weighted \L ojasiewicz-Simon inequality for maps which are close to a fixed harmonic map modulo scaling, we prove Topping's conjecture on continuity of the body map in the case of finite time singularities at which only isolated bubbles form. 
\end{abstract}

\maketitle

\thispagestyle{empty}

\section{Introduction}

Let $(\Sigma,g)$ be an oriented Riemannian surface and $u : \Sigma \times \left[ 0, T \right) \to N,$ with $T < \infty,$ a classical solution of the gradient flow of the Dirichlet energy $$E(u):=\half \int_\Si \abs{du}^2 dV_g,$$ i.e. of the harmonic map flow
$$\frac{\partial u}{\partial t} = \T(u),$$
where $\T(u) = \tr_g \nabla du$ is the tension field. 
The classical result \cite{struwehmsurfaces} of Struwe establishes that the only mechanism that prevents the smooth continuation of the flow past $T$ is the concentration of energy on smaller and smaller balls near a finite set of points $S$, resulting in the bubbling off of harmonic spheres.
As $t$ approaches such a finite singular time $T$, the maps $u(t)$ converge weakly in $H^1(\Sigma, N)$ and smoothly away from the finite singular set $S$ to the so-called body map
$$u(T) = \lim_{t \nearrow T} u(t).$$
A well-known conjecture due to Peter Topping \cite[p. 286]{toppingwindingbehavior} asserts that real-analyticity of the target Riemannian manifold $(N,h)$ is sufficient to rule out essential discontinuities in the body map at a finite-time singularity of the harmonic map flow from a 2-dimensional domain.

This conjecture can be seen as an antidote to the discovery, also in \cite{toppingwindingbehavior}, that for a certain (non-real-analytic) target manifold and carefully chosen initial data, the body map at the first singular time does not extend continuously over the singular point.

In this brief paper, we discuss how certain \Loj estimates, which are related to 
estimates obtained by the first-named author in \cite{rupflinharmonic}, can be obtained and used to fulfil a criterion for continuity of the body map established by the second-named author in \cite{waldronweighted}. This will allow us to prove Topping's conjecture for general analytic targets at singular times at which only isolated bubbles form.

As discussed below, the assumption that all bubbles are isolated has a natural formulation in terms of the (outer) energy scale which was introduced in 
 \cite[Definition 4.2]{songwaldron} and whose definition we recall here.
\begin{defn}
Let $x\in \R^2$ and $\rho > 0.$ Given an $H^1$ map $u: \R^2 \to N,$  define the \emph{(outer) energy scale} 
\begin{equation}\label{lambdadef}
\lambda_{\rho, x}(u) = \inf \left\{ \lambda \in [0,\rho] \,\, \vrule \,\, \sup_{\lambda<r<\rho} E( u, \D_r(x)\setminus \D_{r/2}(x) ) < \eps_0 \right\}.
\end{equation}
Here $\eps_0 > 0$ is a constant  which depends only on the geometry of $N$ and which is chosen small enough so that $\eps$ regularity results such as \cite[Theorem 3.4]{songwaldron} are applicable.  
Further define
\begin{equation*}
    \underline{\lambda}_{\rho, x}(u) = \inf_{y \in \D_{\rho / 2}(x)} \lambda_{\rho,y}(u).
\end{equation*}
\end{defn}

Our main result can then be formulated as follows. 

\begin{thm}\label{thm:main} Let $(N, h)$ be a compact real-analytic manifold. Let $u : \R^2 \times \left[ 0, T \right) \to N$ be a classical solution of harmonic map flow with uniformly bounded Dirichlet energy.
Fix $x \in \R^2$ and suppose that
\begin{equation}\label{main:assumption}
    \limsup_{t \nearrow T} \underline{\lambda}_{\sqrt{T - t}, x}(u(t)) \cdot \sup_{\D_{\sqrt{T - t}}(x)} |du(t)| < \infty.
\end{equation}
Then the body map $u(T)$ is continuous at $x$ and there is no neck between $u(T)$ and any bubble $\om:\R^2\to N$ that forms at $x,$ i.e. we have $u(T,x)=\om(\infty).$
\end{thm}

Before proving this result, we first explain why the theorem above corresponds to singular points at which only a single bubble form. To this end, we recall that the work of Lin-Wang \cite{linwangenergyidentity} and Ding-Tian \cite{dingtian} establishes that there exists a sequence of times $t_i \nearrow T$ such that near each singular point $x\in S$, $u(t_i)$ enjoys a so-called ``bubble-tree decomposition'' into the base map and a finite number of non-constant harmonic maps $\om^j:\R^2\to N$, called bubbles, and that in this convergence there is no energy loss. 
It is well known that using simple parabolic estimates (e.g. \cite[Theorem 3.4]{songwaldron}), such a decomposition can be carried out along an arbitrary sequence of times $t_i \nearrow T$ after passing to a subsequence. We also recall that the points $x_i^k\to x$ and scales $r_i^k\to 0$ at which these bubbles form are so that 
\beq\label{est:bubbling}
\frac{r_i^k}{r_i^l} + \frac{r_i^l}{r_i^k} + \frac{\abs{x_i^k - x_i^l}^2}{r_i^k r_i^l} \to \infty ,
\eeq
as $i \to \infty$ for each $k \neq l$. 
We refer to the paper of Topping \cite{toppingannals} for a summary of this theory. 

We furthermore recall that standard local energy estimates imply that the points and scales at which the bubbles form along the flow are so that $\abs{x_i^k-x}+r_i^k=o(\sqrt{T-t_i})$ and that smallness of energy $E(u(t),B_r(y))\leq \eps_0 $
 yields pointwise control on the gradient of $\abs{du(y,t)}\leqs r^{-1}$. 
Thus the second factor in the main assumption \eqref{main:assumption}  scales like $(\min_k r_i^k)^{-1}$ along $t_i\upto T$. At the same time, if multiple bubbles were to form at $x$ along $t_i\upto T$  then \eqref{est:bubbling} would ensure that $\underline{\lambda}_{\sqrt{T - t_i}, x}(u(t_i))(\min r_i^k)^{-1}\to \infty$. Thus, the main assumption \eqref{main:assumption} is indeed equivalent to asking that no more than one bubble forms at the point $x$.

The new ingredient in the proof of the main theorem is a uniform \L ojasiewicz(-Simon) inequality, Proposition \ref{prop:Loj-L^2}, for maps $\R^2 \to N$ which are constant outside a fixed ball and close to a fixed harmonic map modulo scaling and translation. This estimate follows by a version of some arguments from \cite{rupflinharmonic}, compare also \cite{malchiodirupflinsharp}; in fact, the simpler estimate required here allows for a much easier proof than what was needed in \cite{malchiodirupflinsharp,rupflinharmonic} thanks to the presence of additional symmetries, and we carry out this short proof in \S \ref{sec:weightedestimates}. The \L ojasiewicz inequality of Proposition \ref{prop:Loj-L^2} readily leads to a Gaussian-weighted \L ojasiewicz inequality, Theorem \ref{thm:weightedloj}, for general maps which are close on a ball to a fixed harmonic map modulo scaling. This weighted estimate is exactly the input required to apply the main result of \cite{waldronweighted}.

\subsection*{Acknowledgements}
MR was partially supported by the grant \textit{UKRI3754: Quantitative theory for geometric variational problems}
from the Engineering and Physical Sciences Research
Council. 
SW was supported by the Engineering and Physical Sciences Research
Council studentship 2879243.

\subsection*{AI Statement.} AI was not used in any part of this work. 

\section{$L^2$-\L ojasiewicz estimates near simple bubble trees on $\mathbb{R}^2$}\label{sec:weightedestimates}

We consider maps
$ u : \mathbb{R}^2 \to N\hookrightarrow \R^K$ into 
an analytic manifold $N$
 which are $L^\infty$ and $H^1$ close to a simple bubble tree consisting of a constant body-map and a single bubble and which are constant outside of a fixed size ball. 
 
Denoting the corresponding bubble by $\omega^*:S^2\to N$, we want to show that for such maps the energy 
defect
$
E(u)-E(\omega^*)
$
is controlled in terms of the $L^2$-norm of the tension.

To get such a \L ojasiewicz estimate, we can use that a slight modification of Simon's proof from \cite{leonsimon} gives \L ojasiewicz estimates which are valid with the $H^{-1}$-norm on the right-hand side instead of the $L^2$-norm and in an $L^\infty \cap H^1$ neighbourhood instead of a $C^{2,\al}$ neighbourhood of each critical point.

Namely, we can use

\begin{prop}(Variant of \cite[Theorem $3$]{leonsimon})\label{prop:Loj-Simon}
Let $N$ be any closed analytic manifold and let
$\om^* : S^2 \to N$ be any harmonic map.
Then there exist constants $\eps_1>0$, $C<\infty$, and $\ga\in(1,2]$ such that
\beq\label{est:Loj-Simon}
\abs{E(u)-E(\om^*)}
\leq
C \, \norm{dE(u)}_{H^{-1}(S^2, g_{S^2})}^\ga
\eeq
holds for all $u\in H^1(S^2,N)$ with
$\norm{u-\om^*}_{L^\infty(S^2)}
+\norm{\na(u-\om^*)}_{L^2(S^2 , g_{S^2})}
<\eps_1$.
\end{prop}
As usual, the exponent $\ga$ appearing in this proposition is equal to $2$ if $\om^*$ has no
non-integrable Jacobi fields and otherwise is a (in general not explicitly
known) number $\ga\in (1,2)$.

We note that while the first term in 
$
\norm{f}_{H^1(S^2, g_{S^2})}^2
=
\norm{\na f}_{L^2(S^2, g_{S^2})}^2
+
\norm{f}_{L^2(S^2, g_{S^2})}^2
$ is
conformally invariant, the second one is not. Hence the corresponding
dual norm $\norm{\cdot}_{H^{-1}(S^2, g)}$ of $dE(u)$ is not conformally invariant and neither is the $L^2(S^2,g)$ norm of the tension which we can compute extrinsically as $\tau_g(u)=-\na^{L^2(S^2,g)}E=\Delta_g u+A(u)(\na_gu,\na_g u)$ by viewing $N\hookrightarrow \R^K$.

We now want to use this to show

\begin{prop}\label{prop:Loj-L^2}
Let $\om^* : S^2 \to N$ be a harmonic map into a closed analytic manifold.
Then there exist constants $\eps\in(0,\eps_1]$ and $C<\infty$ 
such that the following holds true for any map 
$u:\R^2\to N$ which
is constant outside of $\D_2$.

Suppose that 
there exist $a\in\D_1$ and a factor $\mu\geq 1$ such that
\beq
\label{ass:Loj-L2}
\norm{\na(u-\om^*\circ \pi_{a,\mu})}_{L^2(\R^2)} + \norm{u-\om^*\circ \pi_{a,\mu}}_{L^\infty(\R^2)}
<\eps \text{ for } \pi_{a,\mu}(x):=\pi\bigl(\mu(x-a)\bigr)
\eeq
where $\pi:\R^2\to S^2$ denotes the inverse stereographic projection. Then
\beq \label{claim:Loj-L2}
\abs{E(u)-E(\om^*)}
\leq
C\,\norm{\tau(u)}_{L^2(\R^2)}^{\ga}
\Bigl(1+\abs{\log(\norm{\tau(u)}_{L^2(\R^2)})}\Bigr)^{\ga/2}.
\eeq
Here $\gamma\in (1,2]$ and $\eps_1>0$ are as in Proposition \ref{prop:Loj-Simon}
\end{prop}

Here and in the following $\tau(u)=\Delta u+A(u)(\na u,\na u)$ denotes the tension field with respect to the standard
metric $dx^2$ and norms on subsets of $\R^2$ are computed with respect to this
metric unless specified otherwise.

\begin{rmk} \label{rmk:mu-large}
We note that condition \eqref{ass:Loj-L2} implies that  for $\eps>0$ chosen sufficiently small the factor $\mu$ is guaranteed to be larger than any fixed constant, as desired in our analysis of maps that are close to simple bubble trees. We also note that the claimed estimate is trivially true for maps whose tension is bounded away from zero and that the claim is invariant under translation. In the following, we can hence restrict our attention to maps with
$\norm{\tau(u)}_{L^2(\R^2)}\leq \de_1$
which are constant outside of $\D_3$ and satisfy \eqref{ass:Loj-L2} for $a=0$ and $\mu\geq \mu_0$, where $\de_1>0$ and $\mu_0 >1$ are sufficiently small/large constants that are chosen below and that only depend on $N$ and $\om^*$. 
\end{rmk}

\begin{proof}[Proof of Proposition \ref{prop:Loj-L^2}]
As the $L^\infty$ and $\dot H^1$ norms are conformally invariant the assumption \eqref{ass:Loj-L2} ensures that  
$ \widetilde u:=u\circ \pi_{a,\mu}^{-1}:S^2\to N$
satisfies the assumptions of Proposition~\ref{prop:Loj-Simon}. We hence deduce that 
\beqa \label{est:Loj-dual}
\abs{E(u)-E(\om^*)}
&=
\abs{E(\widetilde u)-E(\om^*)}
\leq
C \norm{dE(\widetilde u)}_{H^{-1}(S^2, g_{S^2})}^\ga
=
C \norm{dE(u)}_{*,(a,\mu)}^\ga,
\eeqa
where $\norm{\cdot}_{*,(a,\mu)}$ denotes the dual norm to the weighted
$H^1$ norm
\beq \label{def:weighted-H1}
\norm{v}_{(a,\mu)}
:=
\norm{v}_{H^1(\R^2,g_{a,\mu})} \text{ for }
g_{a,\mu}
=
\pi_{a,\mu}^*g_{S^2}.
\eeq

We note that
\beqs
\norm{v}_{L^2(\D_2)}
\leq C
(\log\mu)^\half \cdot \norm{v}_{(a,\mu)},
\eeqs
compare e.g. (2.27) of \cite{rupflinharmonic} and hence that
\beqs
C\norm{\tau(u)}_{L^2(\D_2)}
\geq
(\log\mu)^{-1/2}\norm{dE(u)}_{*,(a,\mu)}. 
\eeqs
The estimate \eqref{est:Loj-dual} hence allows us to deduce that 
\beq \label{est:Loj-with-logmu}
\abs{E(u)-E(\om^*)}
\leq
C(\log\mu)^{\ga/2}
\norm{\tau(u)}_{L^2(\R^2)}^\ga
\eeq
holds 
for any $u$ for which the assumptions of  Proposition~\ref{prop:Loj-L^2} are satisfied for $\eps=\eps_1$.

For maps $u$ for which $\log \mu$ is bounded by some fixed multiple of $\abs{\log(\norm{\tau(u)}_{L^2(\R^2)})}$, so in particular whenever
$
\mu
\leq
\norm{\tau(u)}_{L^2(\R^2)}^{-8}
$, 
we hence immediately get the estimate \eqref{claim:Loj-L2} claimed
in Proposition~\ref{prop:Loj-L^2}.

It thus remains to consider the case where 
$\mu
\geq
\norm{\tau(u)}_{L^2(\R^2)}^{-8}$. In this case
 we should expect to obtain stronger estimates since the
$L^2$ norm of the tension strongly weighs the
highly concentrated region while the contribution of the bulk should be controlled quadratically in
terms of the tension since constant maps are
non-degenerate critical points of the energy.
Indeed, we can show

\textbf{Claim:}
There exists $\eps\in (0,\eps_1]$ and $C$ so that 
the  estimate 
\beqs
\abs{E(u)-E(\om^*)}
\leq
C\norm{\tau(u)}_{L^2(\R^2)}^2
\eeqs 
holds for any map $u$ which satisfies the assumptions
of Proposition~\ref{prop:Loj-L^2} for some 
\beq\label{ass:case2}
\mu \geq \norm{\tau(u)}_{L^2(\R^2)}^{-8}.
\eeq

For the proof of this claim we can use 
standard results on the energy
on neck regions of almost harmonic maps, as found e.g. in Lemma~2.9 of \cite{toppingannals}. These imply that there exist constants $\de_{1,2}>0$ and $K$ depending only on $N$ so that for maps $v:\D_1\to N$ with $\norm{\tau(u)}_{L^2(\D_1)}\leq\de_1 $ 
the estimate 
\beq
E\bigl(v,\D_{e r}\setminus \D_{r}\bigr)
\leq
K r
\label{est:neck-Peter}
\eeq 
holds for all $r\in(0,e^{-1})$ for which  
\beq\label{ass:E-small-neck}
E\bigl(v,\D_1\setminus \D_{r^2}\bigr)<\delta_2. 
\eeq

\begin{proof}[Proof of claim]
As discussed above we can assume that $a=0$ and that $\norm{\tau(u)}_{L^2(\R^2)}\leq \de_1$  where $\de_1=\de_1(N)>0$ is as above. For $\eps$ sufficiently small we immediately get from 
\eqref{ass:Loj-L2} that the 
smallness of the energy required in \eqref{ass:E-small-neck} holds for any radius $r\geq \La \mu^{-1/2}$, $\La$ a sufficiently large constant that only depends only on $N$ and $\om^*$. 
We can hence obtain a radius $r_1\sim \mu^{-1/2}$ for which
\beq \label{est:energy-A-larger}
E\bigl(u,\D_{8r_1}\setminus \D_{r_1/8}\bigr)
\leqs
\mu^{-1/2},
\eeq
where here and in the following we write $\leqs$  if the corresponding inequality holds up to a multiplicative constant $C$ that only depends on $\om^*$ and $N$ and use $\sim $ if both $\leqs $ and $\geqs $ hold.  

Setting
$
A:=\D_{2r_1}\setminus \D_{r_1/2}
$ and $ m:=\fint_A u
$ we can now use that 
\beqs
\norm{u-m}_{L^\infty(A)}\leq \osc_A u
\leq
\osc_A(\om^*\circ\pi_\mu)
+
2\norm{u-\om^*\circ\pi_\mu}_{L^\infty}
\leqs
(r_1\mu)^{-1/2}+\eps\leq \mu^{-1/2}+\eps
\eeqs
can be made smaller than any given positive constant by choosing $\eps$ sufficiently small, compare Remark \ref{rmk:mu-large}. In particular, $m$ will be contained in a small tubular neighbourhood of $N$ on which the nearest point projection $\pi_N$ is well defined and smooth.
Fixing cut-off functions 
$\varphi_{1,2}$ with  
\beqs
\varphi_1\equiv1 \text{ on }\D_{r_1}
\text{ and }
\varphi_1\equiv0 \text{ on }(\D_{2r_1})^c,\text{ while } 
\varphi_2\equiv1 \text{ on }\R^2\setminus\D_{r_1}
\text{ and }
\varphi_2\equiv0 \text{ on }\D_{r_1/2}
\eeqs
and with $|\na\varphi_i|^2+|\na^2\varphi_i|
\leqs
r_1^{-2}$, we 
can hence get maps
$$
u_1=\pi_N\bigl(\varphi_1u+(1-\varphi_1)m\bigr) :\R^2\to N \text{ and }
u_2:=\varphi_2u+(1-\varphi_2)m:\R^2\to \R^K\hookleftarrow N
$$
with $
u_1\equiv u$ on $\D_{r_1}$
and  $ u_1\equiv \pi_N(m)$ on $
\R^2\setminus\D_{2r_1},
$ while 
$ u_2\equiv u$ on $ \R^2\setminus\D_{r_1} 
$ 
and $
u_2\equiv m$ on $ \D_{r_1/2},
$ for which 
\beqas
E(u_1)+E(u_2)-E(u)
&\leqs
E(u,A)
+
r_1^{-2} \norm{u-\fint_A u}_{L^2(A)}^2
\leqs
\norm{\na u}_{L^2(A)}^2
\leqs
\mu^{-1/2}
\leqs
\norm{\tau(u)}_{L^2(\R^2)}^4.
\eeqas
The claim hence follows once we show that 
both
\beq
\abs{E(u_1)-E(\om^*)}
\leqs
\norm{\tau(u)}_{L^2(\R^2)}^2
\label{claim:est1}
\eeq
and
\beq
E(u_2)
\leqs
\norm{\tau(u)}_{L^2(\R^2)}^2.
\label{claim:est2}
\eeq

\textbf{Proof of \eqref{claim:est1}.}\\
We first note that as $(r_1 \mu)^{-1}\sim \mu^{-1/2}$ we can bound
$$ 
\norm{u_1-\om^*\circ\pi_{\mu}}_{L^\infty(\R^2)} \leqs \osc_A u + \norm{u-\om^*\circ\pi_{\mu}}_{L^\infty(\R^2)}+\osc_{\R^2\setminus \D_{r_1}}\pi_\mu\leqs  \eps+\mu^{-1/2} 
$$ and 
$$\norm{\na (u_1-\om^*\circ\pi_{\mu})}_{L^2(\R^2)}
\leqs E(u,A)^\half+\norm{\na (u-\om^*\circ\pi_{\mu})}_{L^2(\R^2)} +\norm{\na \pi_\mu}_{L^2(\R^2\setminus \D_{r_1})}\leqs \mu^{-1/4}+\eps.
$$
A sufficiently small choice of $\eps$ hence ensures that $\norm{u_1-\om^*\circ\pi_{\mu}}_{L^\infty(\R^2)}+\norm{\na (u_1-\om^*\circ\pi_{\mu})}_{L^2(\R^2)}\leq \eps_1$.
As these quantities remain invariant as we rescale $\D_{r_1}$ to $\D_1$ we hence deduce that $\tilde u_1:=u_1(r_1 x)$, which is by construction constant outside of $\D_2$, satisfies the assumptions of the proposition with $\tilde \mu:= r_1 \mu\sim \mu^{-1/2}$ in place of $\mu$ and $\eps=\eps_1$. This allows us to apply \eqref{est:Loj-with-logmu} for $\tilde u_1$ and $\tilde \mu$ to deduce that  
\beq \label{est:E-u1}
\abs{E(u_1)-E(\om^*)}=\abs{E(u_1)-E(\om^*)}
\leqs
(\log\tilde \mu)^{\ga/2}
\norm{\tau(\tilde u_1)}_{L^2(\R^2)}^\ga.
\eeq
We now crucially use that the above rescaling significantly reduces the $L^2$ norm of the tension, namely that 
\beq 
\label{est:tension-tilde}
\norm{\tau(\tilde u_1)}_{L^2(\R^2)}=r_1\norm{\tau(u_1)}_{L^2(\R^2)}\leq r_1\norm{\tau(u)}_{L^2(\R^2)}+r_1 \norm{\tau(u_1)}_{L^2(A)}.\eeq 
As the energy of $u$ is small on the larger annulus $A_2:=\D_{4r_1}\setminus\D_{r_1/4}$ standard $H^2$ estimates for almost harmonic maps as found e.g. in \cite[Lemma $2.1$]{dingtian} imply that 
\beqs
\int_{A} |\na^2u|^2+|\na u|^4
\leqs
r_1^{-2}E(u,A_2)
+
\norm{\tau(u)}_{L^2(A_2)}^2.
\eeqs
Combined with the trivial pointwise bound $\abs{\tau(u_1)}\leqs \abs{\na^2 u}+\abs{\na u}^2+r_1^{-2} \abs{u-m}+r_1^{-1}\abs{\na u} \mathbbm{1}_A$ and the Poincar\'e inequality this allows us to deduce that 
$$\norm{\tau(u_1)}_{L^2(A)}\leqs r_1^{-1} E(u,A_2)^{1/2}+\norm{\tau(u)}_{L^2(A_2)}.$$
As $r_1\sim \mu^{-1/2}$ while $E(u,A_2)\leqs \mu^{-1/2}$, compare \eqref{est:energy-A-larger}, we can hence bound  
$$r_1\norm{\tau(u_1)}_{L^2(A)}\leqs  E(u,A_2)^{1/2}+r_1\norm{\tau(u)}_{L^2(A_2)}\leqs \mu^{-1/4}+\mu^{-1/2} \de_1\leqs \mu^{-1/4}.$$
Inserting this, \eqref{est:tension-tilde} and  $\log\tilde\mu\sim \log \mu$ into \eqref{est:E-u1} hence yields the desired bound of   
$$\abs{E(u_1)-E(\om^*)}
\leqs
(\log \mu)^{\ga/2} \mu^{-\gamma/4}\leqs \mu^{-1/4}\leqs \norm{\tau(u)}_{L^2(\R^2)}^2$$
since $\gamma>1$ and since $\mu$ satisfies \eqref{ass:case2}. 
\end{proof}

\paragraph{\textbf{Proof of \eqref{claim:est2}.}}
$ $\\
As $\na \varphi_2$ is supported on $A$ and  $u_2=u+(1-\varphi_2)(m-u)$, we have
pointwise bounds of 
\beqas
|\Delta u_2|
&\leq\varphi_2
|\Delta u|+2
|\na\varphi_2||\na u|+
|\Delta\varphi_2||u-m|
\\
&\leqs
\varphi_2|\tau(u)|
+
\varphi_2|\na u|^2
+
\mathbbm{1}_A\bigl(r_1^{-1}|\na u|+r_1^{-2}|u-m|\bigr)\\
&\leqs
|\tau(u)|
+
|\na u_2|^2
+
\mathbbm{1}_A\bigl(|\na u|^2+r_1^{-1}|\na u|+r_1^{-2}|u-m|\bigr).
\eeqas
As $u$, and thus $u_2$, are constant outside of $\D_3$ we can combine the above estimate with the Poincar\'e inequality, applied on both $\D_3$ and on $A$, to estimate
\beqas
\int_{\R^2}|\na u_2|^2&=\int_{\D_3}|\na u_2|^2
=
-\int_{\D_3}
\bigl(u_2-\fint_{\D_3}u_2\bigr)\cdot\Delta u_2
\\
&\leqs
\norm{\tau(u)}_{L^2(\R^2)}
\norm{\na u_2}_{L^2(\D_3)}
+\osc(u_2) \int_{\D_3} \abs{\na u_2}^2+E(u,A).
\eeqas
As $\osc(u_2)\leqs
\osc_{\D_{r_1}^c}(\om^*\circ\pi_\mu)+\norm{u-\om^*\circ\pi_{\mu}}_{L^\infty}\leqs (\mu r_1)^{-1}+\eps\leqs \mu^{-1/2}+\eps$ is guaranteed to be small if $\eps$ is small, we can  absorb the second term 
to deduce that
\beqas
\int_{\R^2}|\na u_2|^2
&\leqs
\norm{\tau(u)}_{L^2(\R^2)}^2 
+
E(u,A)\leqs \norm{\tau(u)}_{L^2(\R^2)}^2 +\mu^{-1/2} \leqs\norm{\tau(u)}_{L^2(\R^2)}^2, 
\eeqas
using \eqref{est:energy-A-larger} in the second and 
\eqref{ass:case2} in the last step. 

This establishes \eqref{claim:est2} and hence completes the proof of the proposition.
\end{proof}

\section{Weighted \L ojasiewicz estimate and proof of the main theorem}
For $\tau > 0,$ define the Gaussian-weighted $L^2$-norm:
\begin{equation*}
\| \alpha \|_\tau = \sqrt{ \int_{\R^2} |\alpha |^2 \, e^{-\frac{r^2}{4\tau} } \, dV }.
\end{equation*}
For $u : \R^2 \to N,$ also let
$$\Phi_\tau(u) = \frac12 \int |du|^2 e^{ \frac{-r^2}{4 \tau} } \, dV = \frac12 \| du \|_{\tau}^2$$
and
$$\That_\tau(u) = \T(u) - \frac{x}{2 \tau} \intprod du.$$

\begin{thm}\label{thm:weightedloj} Let $(N, h)$ be a compact real-analytic manifold and let $E_0, K > 0$ be given. There exists a constant $C$, an exponent $\gamma\in (1,2]$ and a finite set of critical values $\mathcal{E}_0 \subset [ 0, E_0 ]$ of the Dirichlet energy, depending only on $N, K,$ and $E_0,$ so that the following holds.

Let $u : \R^2 \to N$ be any map with $E(u) \leq E_0.$ Suppose that there exists $\tau > 0$ such that
\begin{equation}\label{weightedloj:assumption}
    \underline{\lambda}_{\tau, 0}(u) \sup_{\D_{\tau}(0)} |du| \leq K.
\end{equation}
Then there exists $E_1 \in \mathcal{E}_0$ such that $u$ obeys a \L ojasiewicz inequality of the form
\begin{equation}\label{weightedloj:loj}
    |\Phi_\tau(u) - E_1 | \leq C \left( 1 + \abs{\log \left( \sqrt{\tau} \| \That(u) \|_\tau \right)} \right)^{\gamma / 2} \left( \sqrt{\tau} \| \That(u) \|_\tau \right)^\gamma.
\end{equation}
\end{thm}
\begin{proof} Rescaling by $\sqrt{\tau},$ we can assume without loss of generality that $\tau = 1.$ The assumption (\ref{weightedloj:assumption}) becomes
\begin{equation}\label{weightedloj:scaledassumption}
    \underline{\lambda}_{1, 0}(u) \sup_{\D_{1}(0)} |du| \leq K.
\end{equation}
Let now $\mathcal{H}_0$ be the set of  harmonic maps $\omega: \R^2 \to N$ that satisfy 
\begin{equation}\label{weightedloj:definitionofscrE0}
    \underline{\lambda}_{R, 0}(\omega) \sup_{\D_{R}(0)} |d\omega| \leq K
\end{equation}
for all $R \geq 1.$
We note that this set $\mathcal{H}_0$ is smoothly compact modulo scaling. Since $N$ is real-analytic, the standard \L ojasiewicz-Simon inequality \cite{leonsimon} implies that the corresponding set of critical values
$$\mathcal{E}_0=\lbrace E(\om): \om\in \HH_0 \rbrace$$
is finite.

Let $u$ be any map as in the statement. First note that the weighted Poincar\'e inequality of \cite[Lemma 2.1]{waldronweighted} gives
\begin{equation}\label{Thatcontrolsr2duetc}
\int r^2 |du|^2 e^{-\frac{r^2}{4}}  + \| \T(u) \|^2_{L^2(\D_2)} \leq C \| \That(u) \|_1^2.
\end{equation}
In particular, we have
\begin{equation}\label{energycontrolonannulus}
E(u, \D_2 \setminus \D_1) \leq C\| \That(u) \|_1^2.
\end{equation}
Similarly, on any compact subset of $\R^2,$ the $L^2$-norm of the tension of $u$ is controlled by the $L^2$-norm of the weighted tension.

We now argue by contradiction: supposing that the result is false, there must exist a sequence of maps $u_i : \R^2 \to N$ as in the statement
but with
\begin{equation}\label{weightedloj:contradloj}
    \frac{|\Phi_1(u_i) - E_1 |}{i} > \left( 1 + \abs{\log  \| \That(u_i) \|_1 }\right)^{\frac12 + \frac{1}{2i}} \| \That(u_i) \|_1^{1 + \frac{1}{i}}
\end{equation}
for each $E_1 \in \mathcal{E}_0.$ In particular,
\beq 
\label{ass:tension-to-zero}\| \That(u_i) \|_1 \to 0\eeq
as $i \to \infty.$

Let $R \geq 1.$ In view of 
\eqref{Thatcontrolsr2duetc} and \eqref{ass:tension-to-zero} we clearly have
\begin{equation*}
    \lambda_{R,0}(u_i) \to 0
\end{equation*}
as well as
\begin{equation*}
    \lambda_{R,0}(u_i) = \lambda_{1,0}(u_i)
\end{equation*}
for $i$ sufficiently large. Thus (\ref{weightedloj:scaledassumption}) gives
\begin{equation}\label{weightedloj:uiassumption}
    \underline{\lambda}_{R, 0}(u_i) \sup_{\D_{R}(0)} |du_i|  \leq K
\end{equation}
for $i$ sufficiently large.

In view of the tension control, we can pass to a subsequence that converges in the bubble-tree sense on $\R^2.$ Due to (\ref{Thatcontrolsr2duetc}), the origin is the only possible energy-concentration point. 
Moreover, due to (\ref{weightedloj:assumption}), there is only a single bubble $\omega^*:S^2\to N$ at the origin. The standard bubbling analysis for almost-harmonic maps provides $\mu_i \to \infty$ and $a_i \to 0$ such that
\beq
\label{weightedloj:Loj-L2}
\norm{\na(u_i-\om^*\circ \pi_{a_i,\mu_i})}_{L^2(\D_2(0))} + \norm{u-\om^*\circ \pi_{a_i,\mu_i}}_{L^\infty(\D_2(0))}
\to 0.
\eeq
Furthermore, the bound (\ref{weightedloj:uiassumption}) implies that $\om^*\circ \pi_{a_i,\mu_i}$ satisfies (\ref{weightedloj:definitionofscrE0}); hence $E(\om^*\circ \pi_{a_i,\mu_i}) = E(\omega^*)$ belongs to $\mathcal{E}_0.$

We now cut off $u_i$ so that it is constant outside of $\D_2$ as required to apply Proposition \ref{prop:Loj-L^2}. For this we can proceed as in the proof of Proposition \ref{prop:Loj-L^2} above, compare also \cite[Lemma 3.7]{waldronlojonS2} and Topping \cite[Lemma 2.23]{toppingannals}, to obtain a map $w_i$ that agrees with $u_i$ on $\D_1,$ is constant outside $\D_2,$ and satisfies
\begin{equation}\label{Ewicontrol}
E(w_i, \D_2 \setminus \D_1) + \|\T(w_i) \|_{L^2(\R^2)}^2 \lesssim E(u_i, \D_2 \setminus \D_1) + \|\T(u_i) \|_{L^2(\R^2)}^2 \lesssim \| \That(u_i)\|_{1}^2,
\end{equation}
as well as (\ref{weightedloj:Loj-L2}) with $w_i$ in place of $u_i.$ Applying Proposition \ref{prop:Loj-L^2} to $w_i,$ we obtain
\begin{equation*}
    |E(w_i) - E(\omega^*)| \lesssim \,\norm{\T(w_i)}_{L^2(\R^2)}^{\ga}
\Bigl(1+\abs{\log(\norm{\T(w_i)}_{L^2(\R^2)})}\Bigr)^{\ga/2}
\end{equation*}
for an exponent $\gamma > 1.$
Exchanging $w_i$ for $u_i$ using (\ref{energycontrolonannulus}) and (\ref{Ewicontrol}), we get
\begin{equation*}
    |E(u_i, \D_1) - E(\omega^*)| \lesssim \,\norm{\That(u_i)}_{1}^{\ga}
\Bigl(1+\abs{\log(\norm{\That(u_i)}_{1}}\Bigr)^{\ga/2}.
\end{equation*}
We further have
$$\left| E(u_i, \D_1) - \frac12 \int_{\D_1} |du_i|^2 e^{\frac{-r^2}{4}} \right| \leq \int_{\D_1} r^2 |du_i|^2 \lesssim \| \That(u_i) \|_1^{2}.$$
Finally, we have
$$\int_{\R^2 \setminus \D_1} |du_i|^2 e^{\frac{-r^2}{4}} \lesssim \int_{\R^2 \setminus \D_1} r^2 |du|^2 e^{\frac{-r^2}{4}} \lesssim \| \That(u_i) \|_1^{2}.$$
Combining the last three estimates, we obtain (\ref{weightedloj:loj}) for $u = u_i$ and $E_1 = E(\omega^*).$ But for $i$ sufficiently large, this contradicts (\ref{weightedloj:contradloj}).
\end{proof}

\begin{proof}[Proof of Theorem \ref{thm:main}]
We can take $x = 0$ without loss of generality. The assumption (\ref{main:assumption}) implies (\ref{weightedloj:assumption}), with $\tau = \sqrt{T - t},$ for times $t$ sufficiently close to $T.$ By the previous theorem, for each such $t,$ the \L ojasiewicz inequality (\ref{weightedloj:loj}) is satisfied for $u = u(t).$ Since the quantity $\Phi_{\sqrt{T - t}}(u(t))$ is monotone-decreasing (see \cite{struwehmsurfaces} or \cite[Theorem 3.2]{waldronweighted}) and the set $\mathcal{E}_0$ is finite, by further advancing $t,$ we can assume that (\ref{weightedloj:loj}) is satisfied for the same value $E_1 \in \mathcal{E}_0$ at all times. The hypothesis of \cite[Corollary 4.3]{waldronweighted} is therefore satisfied, giving the desired conclusions.
\end{proof}

\bibliographystyle{abbrv}
\bibliography{biblio}

\textsc{Melanie Rupflin} and \textsc{Sebastian Woodward}, Mathematical Institute, University of Oxford.

\textsc{Alex Waldron}, University of Wisconsin, Madison.

\end{document}